\documentclass[12pt]{article}
\usepackage{amsmath,amssymb}    

\makeatletter
\newfont{\footsc}{cmcsc10 at 8truept}
\newfont{\footbf}{cmbx10 at 8truept}  
\newfont{\footrm}{cmr10 at 10truept} 
\makeatother
\newtheorem{theorem}{Theorem}

\newtheorem{Cor}{Corollary}

\newcommand{\qed}{\hfill \rule{0.7ex}{1.5ex}}

\newenvironment{proof}{\begin{trivlist} \item[\hskip \labelsep{\it
Proof.}]\setlength{\parindent}{0pt}}{\end{trivlist}}

\title{$q$-poly-Bernoulli numbers and $q$-poly-Cauchy numbers with a parameter by Jackson's integrals}

\author{Takao Komatsu
\\   
\small School of Mathematics and Statistics\\[-0.8ex]
\small Wuhan University, Wuhan, 430072, China\\[-0.8ex]
\small \texttt{komatsu@whu.edu.cn}
}

\date{
\small MR Subject Classifications: 05A15, 05A19, 11B75, 33D45}

\begin{document}
\maketitle

\begin{abstract}
We define $q$-poly-Bernoulli polynomials $B_{n,\rho,q}^{(k)}(z)$ with a parameter $\rho$, $q$-poly-Cauchy polynomials of the first kind $c_{n,\rho,q}^{(k)}(z)$ and of the second kind $\widehat c_{n,\rho,q}^{(k)}(z)$ with a parameter $\rho$  by Jackson's integrals, which generalize the previously known numbers and polynomials, including poly-Bernoulli numbers $B_n^{(k)}$ and the poly-Cauchy numbers of the first kind $c_n^{(k)}$ and of the second kind $\widehat c_n^{(k)}$. We investigate their properties connected with usual Stirling numbers and weighted Stirling numbers. We also give the relations between generalized poly-Bernoulli polynomials and two kinds of generalized poly-Cauchy polynomials.
\end{abstract}

\section{Introduction}

Let $n$ and $k$ be integers with $n\ge 0$, and let $\rho$ be a real number parameter\footnote{$q$ was used as a parameter in \cite{Ko2,KC}, but in this paper we use $\rho$ in order to avoid confusions with $q$-integral.} with $\rho\ne 0$. Let $q$ be a real number with $0\le q<1$.  
Define $q$-poly-Bernoulli polynomials $B_{n,\rho,q}^{(k)}(z)$ with a parameter $\rho$ by 
\begin{equation} 
\frac{\rho}{1-e^{-\rho t}}{\rm Li}_{k,q}\left(\frac{1-e^{-\rho t}}{\rho}\right)e^{-t z}=\sum_{n=0}^\infty B_{n,\rho,q}^{(k)}(z)\frac{t^n}{n!}\,,
\label{def:qpb} 
\end{equation} 
where ${\rm Li}_{k,q}(z)$ is the $q$-polylogarithm function (\cite{Katsurada}) defined by 
$$
{\rm Li}_{k,q}(z)=\sum_{n=1}^\infty\frac{z^n}{[n]_q^k}\,. 
$$ 
Here, 
$$
[x]_q=\frac{1-q^x}{1-q} 
$$ 
is the $q$-number with $[0]_q=0$ (see e.g. \cite[(10.2.3)]{AAR}, \cite{Ja}).  Note that $\lim_{q\to 1}[x]_q=x$. 

Notice that 
$$
\lim_{q\to 1}B_{n,\rho,q}^{(k)}(z)=B_{n,\rho}^{(k)}(z)\,,
$$ 
which is the poly-Bernoulli polynomial with a $\rho$ parameter (\cite{KC}), and 
$$
\lim_{q\to 1}{\rm Li}_{k,q}(z)={\rm Li}_k(z)\,,
$$
which is the ordinary polylogarithm function, defined by 
\begin{equation} 
{\rm Li}_k(z)=\sum_{m=1}^\infty\frac{z^m}{m^k}\,. 
\label{polylog}
\end{equation}    
In addition, when $z=0$, $B_{n,\rho}^{(k)}(0)=B_{n,\rho}^{(k)}$ is the poly-Bernoulli number with a $\rho$ parameter. When $z=0$ and $\rho=1$, $B_{n,1}^{(k)}(0)=B_n^{(k)}$ is the poly-Bernoulli number (\cite{Kaneko}) defined by 
\begin{equation} 
\frac{{\rm Li}_k(1-e^{-t})}{1-e^{-t}}=\sum_{n=0}^\infty B_n^{(k)}\frac{t^n}{n!}\,,
\label{Kaneko} 
\end{equation} 
The poly-Bernoulli numbers are extended to the poly-Bernoulli polynomials (\cite{BH,CC}) and to the special multi-poly-Bernoulli numbers (\cite{HM}).

Let $n$ and $k$ be integers with $n\ge 0$, and let $\rho$ be a real number parameter with $\rho\ne 0$. Let $q$ be a real number with $0\le q<1$. Define $q$-poly-Cauchy polynomials of the first kind $c_{n,\rho,q}^{(k)}(z)$ with a parameter $\rho$ by 
\begin{align}  
c_{n,\rho,q}^{(k)}(z)&=\rho^n\underbrace{\int_0^1\cdots\int_0^1}_k\left(\frac{x_1\cdots x_k-z}{\rho}\right)_n d_q x_1\cdots d_q x_k\notag\\
&=\rho^n n!\underbrace{\int_0^1\cdots\int_0^1}_k\binom{(x_1\cdots x_k-z)/\rho}{n}d_q x_1\cdots d_q x_k\,, 
\label{def:qpc1}
\end{align}
where $(x)_n=x(x-1)\cdots(x-n+1)$ ($n\ge 1$) with $(x)_0=1$. 

Jackson's $q$-derivative with $0<q<1$ (see e.g. \cite[(10.2.3)]{AAR}, \cite{Ja}) is defined by 
$$
D_q f=\frac{d_q f}{d_q x}=\frac{f(x)-f(q x)}{(1-q)x}
$$ 
and Jackson's $q$-integral (\cite[(10.1.3)]{AAR}, \cite{Ja}) is defined by 
$$
\int_0^x f(t)d_q t=(1-q)x\sum_{n=0}^\infty f(q^n x)q^n\,. 
$$ 
For example, when $f(x)=x^m$ for some nonnegative integer $m$,
\begin{align*}  
D_q f&=\frac{x^m-q^m x^m}{(1-q)x}\\
&=[m]_q x^{m-1}
\end{align*}  
and 
\begin{align*} 
\int_0^x t^m d_q t&=(1-q)x\sum_{n=0}^\infty q^{mn}x^m q^n\\
&=(1-q)x^{m+1}\sum_{n=0}^\infty q^{n(m+1)}\\
&=\frac{x^{m+1}}{[m+1]_q}\,. 
\end{align*} 

Notice that 
$$
\lim_{q\to 1}c_{n,\rho,q}^{(k)}(z)=c_{n,\rho}^{(k)}(z)\,, 
$$ 
which is the poly-Cauchy polynomial with a $\rho$ parameter (\cite{Ko2}). In addition, when $z=0$, $c_{n,\rho}^{(k)}(0)=c_{n,\rho}^{(k)}$ is the poly-Cauchy number with a $\rho$ parameter (\cite{Ko2}). When $z=0$ and $\rho=1$, $c_{n,1}(0)=c_n^{(k)}$ is the poly-Cauchy number (of the first kind) (\cite{Ko1}), defined by the integral of the falling factorial: 
\begin{align*}
c_n^{(k)}&=\underbrace{\int_0^1\cdots\int_0^1}_k(x_1\cdots x_k)_n dx_1\dots dx_k\\
&=n!\underbrace{\int_0^1\cdots\int_0^1}_k\binom{x_1\cdots x_k}{n}dx_1\dots dx_k\,.
\end{align*}  
If $k=1$, then $c_n^{(1)}=c_n$ is the classical Cauchy number (\cite{C,MSV}).  The number $c_n/n!$ is sometimes referred to as the Bernoulli number of the second kind (\cite{Ca1,Jo,No}). 
The poly-Cauchy numbers of the first kind $c_n^{(k)}$ can be expressed in terms of the Stirling numbers of the first kind.  
$$
c_n^{(k)}=\sum_{m=0}^n\frac{(-1)^{n-m}S_1(n,m)}{(m+1)^k}\quad(n\ge 0,~k\ge 1) 
$$
(\cite[Theorem 1]{Ko1}), where $S_1(n,m)$ is the (unsigned) Stirling number of the first kind, see \cite{C}, determined by the rising factorial: 
\begin{equation}  
x(x+1)\cdots(x+n-1)=\sum_{m=0}^n S_1(n,m)x^m\,. 
\label{def:signedst1}
\end{equation}    
The generating function of the poly-Cauchy numbers $c_n^{(k)}$ is given by 
\begin{equation} 
{\rm Lif}_k\bigl(\ln(1+t)\bigr)=\sum_{n=0}^\infty c_n^{(k)}\frac{t^n}{n!}  
\label{polycauchy}
\end{equation}  
(\cite[Theorem 2]{Ko1}), where ${\rm Lif}_k(z)$ is called {\it polylogarithm factorial} function (or simply, {\it polyfactorial} function) defined by 
\begin{equation} 
{\rm Lif}_k(z)=\sum_{m=0}^\infty\frac{z^m}{m!(m+1)^k}\,. 
\label{polyfac}
\end{equation}   
By this definition, $k$ is not restricted to a positive integer in $c_n^{(k)}$.  
Similarly, define the poly-Cauchy numbers of the second kind $\widehat c_n^{(k)}$ (\cite{Ko1}) by 
\begin{align*}
\widehat c_n^{(k)}&=\underbrace{\int_0^1\cdots\int_0^1}_k(-x_1\cdots x_k)_n dx_1\dots dx_k\\
&=n!\underbrace{\int_0^1\cdots\int_0^1}_k\binom{-x_1\cdots x_k}{n}dx_1\dots dx_k\,.
\end{align*}  
If $k=1$, then $\widehat c_n^{(1)}=\widehat c_n$ is the classical Cauchy number of the second kind (\cite{C,MSV}).  
The poly-Cauchy numbers of the second kind $\widehat c_n^{(k)}$ can be expressed in terms of the Stirling numbers of the first kind.  
$$
\widehat c_n^{(k)}=(-1)^n\sum_{m=0}^n\frac{S_1(n,m)}{(m+1)^k}\quad(n\ge 0,~k\ge 1) 
$$
(\cite[Theorem 4]{Ko1}). 
The generating function of the poly-Cauchy numbers of the second kind $\widehat c_n^{(k)}$ is given by 
\begin{equation} 
{\rm Lif}_k\bigl(-\ln(1+t)\bigr)=\sum_{n=0}^\infty\widehat c_n^{(k)}\frac{t^n}{n!}  
\label{polycauchy2}
\end{equation}  
(\cite[Theorem 5]{Ko1}). 

The poly-Cauchy numbers have been considered as analogues of poly-Bernoulli numbers $B_n^{(k)}$.  
The poly-Cauchy numbers (of the both kinds) are extended to the poly-Cauchy polynomials (\cite{KK}), and to the poly-Cauchy numbers with a $q$ parameter (\cite{Ko2}). The corresponding poly-Bernoulli numbers with a $q$ parameter can be obtained in \cite{KC}. A different direction of generalizations of Cauchy numbers is about Hypergeometric Cauchy numbers (\cite{Ko4}). Arithmetical and combinatorial properties including sums of products have been studied (\cite{Ko3,KLS,KL}). 

In this paper, by using Jackson's $q$-integrals, as essential generalizations of the previously known numbers, including poly-Bernoulli numbers $B_n^{(k)}$, the poly-Cauchy numbers of the first kind $c_n^{(k)}$ and of the second kind $\widehat c_n^{(k)}$, we introduce the concept about $q$-analogues or extensions of the poly-Bernoulli polynomials $B_{n,\rho,q}^{(k)}(z)$ with a parameter, the poly-Cauchy polynomials of the first kind $c_{n,\rho,q}^{(k)}$ and of the second kind $\widehat c_{n,\rho,q}^{(k)}$ with a parameter. We investigate their properties connected with usual Stirling numbers and weighted Stirling numbers. We also give the relations between generalize poly-Bernoulli polynomials and two kinds of generalized poly-Cauchy polynomials.


\section{$q$-poly-Bernoulli polynomials with a parameter} 

Carlitz \cite{Ca2} defined the weighted Stirling numbers of the first kind $S_1(n,m,x)$ and of the second kind $S_2(n,m,x)$ by  
\begin{equation} 
\frac{(1-t)^{-x}\bigl(-\ln(1-t)\bigr)^m}{m!}=\sum_{n=0}^\infty S_1(n,m,x)\frac{t^n}{n!}
\label{ws1}
\end{equation}  
and 
\begin{equation} 
\frac{e^{x t}(e^t-1)^m}{m!}=\sum_{n=0}^\infty S_2(n,m,x)\frac{t^n}{n!}\,, 
\label{ws2} 
\end{equation}  
respectively.   
Note that Carlitz \cite{Ca2} used the notation $R_1(n,m,x)$ and $R(n,m,x)$ instead of $S_1(n,m,x)$ and $S_2(n,m,x)$, respectively.  
When $x=0$, $S_1(n,m,0)=S_1(n,m)$ and $S_2(n,m,0)=S_2(n,m)$ are the (unsigned) Stirling number of the first kind and the Stirling number of the second kind, respectively.  

The $q$-poly-Bernoulli polynomials with a parameter $\rho$ can be expressed in terms of the weighted Stirling numbers of the second kind. 

\begin{theorem}  
We have 
\begin{equation}  
B_{n,\rho,q}^{(k)}(z)=\sum_{m=0}^n S_2\left(n,m,\frac{z}{\rho}\right)\frac{(-\rho)^{n-m}m!}{[m+1]_q^k}\,. 
\label{exp:qpb}
\end{equation}  
\label{thb1}
\end{theorem}  

\begin{proof} 
From (\ref{def:qpb}), and using (\ref{ws2}), we have 
\begin{align*}  
\sum_{n=0}^\infty B_{n,\rho,q}^{(k)}(z)\frac{t^n}{n!}&=\sum_{m=0}^\infty\frac{(-\rho)^{-m}}{[m+1]_q^k}e^{-\rho t(z/\rho)}(e^{-\rho t}-1)^m\\
&=\sum_{m=0}^\infty\frac{(-\rho)^{-m}m!}{[m+1]_q^k}\sum_{n=m}^\infty S_2\left(n,m,\frac{z}{\rho}\right)\frac{(-\rho t)^n}{n!}\\
&=\sum_{n=0}^\infty\left(\sum_{m=0}^\infty S_2\left(n,m,\frac{z}{\rho}\right)\frac{(-\rho)^{n-m}m!}{[m+1]_q^k}\right)\frac{t^n}{n!}\,. 
\end{align*} 
Comparing the coefficients on both sides, we get the result. Notice that $S_2(n,m,x)=0$ for $n<m$. 
\qed\end{proof}

\begin{Cor}  
For $q$-poly-Bernoulli numbers with a parameter $\rho$, we have 
\begin{equation} 
B_{n,\rho,q}^{(k)}=\sum_{m=0}^n S_2(n,m)\frac{(-\rho)^{n-m}m!}{[m+1]_q^k}\,. 
\label{exp:qpbn}
\end{equation}  
\end{Cor}

\section{$q$-poly-Cauchy polynomials of the first kind with a parameter} 

The $q$-poly-Cauchy polynomials of the first kind $c_{n,\rho,q}^{(k)}(z)$ with a parameter can be expressed in terms of the weighted Stirling numbers of the first kind $S_1(n,m,x)$.  In this expression, $k$ is not restricted to a positive integer. 

\begin{theorem}  
For integers $n$ and $k$ with $n\ge 0$, we have 
\begin{align} 
c_{n,\rho,q}^{(k)}(z)&=\sum_{m=0}^n S_1(n,m)(-\rho)^{n-m}\sum_{i=0}^m\binom{m}{i}\frac{(-z)^i}{[m-i+1]_q^k}\notag\\
&=\sum_{m=0}^n S_1\left(n,m,\frac{z}{\rho}\right)\frac{(-\rho)^{n-m}}{[m+1]_q^k}\,.  
\label{exp:qpc1}
\end{align}  
\label{thc1} 
\end{theorem}   

\begin{proof} 
From (\ref{def:qpc1}) and (\ref{def:signedst1}), we have 
\begin{align*}  
c_{n,\rho,q}^{(k)}(z)&=\rho^n\sum_{m=0}^n(-1)^{n-m}S_1(n,m)\underbrace{\int_0^1\cdots\int_0^1}_k\left(\frac{x_1\cdots x_k-z}{\rho}\right)^m d_q x_1\cdots d_q x_k\\
&=\sum_{m=0}^n(-\rho)^{n-m}S_1(n,m)\sum_{i=0}^m\binom{m}{i}(-z)^{m-i}\underbrace{\int_0^1\cdots\int_0^1}_k x_1^i\cdots x_k^i d_q x_1\cdots d_q x_k\\
&=\sum_{m=0}^n(-\rho)^{n-m}S_1(n,m)\sum_{i=0}^m\binom{m}{i}\frac{(-z)^{m-i}}{[i+1]_q^k}\\
&=\sum_{m=0}^n(-\rho)^{n-m}S_1(n,m)\sum_{i=0}^m\binom{m}{i}\frac{(-z)^{i}}{[m-i+1]_q^k}\,. 
\end{align*} 
By using the relation, see \cite[Eq. (5.2)]{Ca2}, 
$$
S_1(n,m,x)=\sum_{i=0}^n\binom{m+i}{i}x^i S_1(n,m+i)\,, 
$$ 
we obtain 
\begin{align*}  
c_{n,\rho,q}^{(k)}(z)&=\sum_{i=0}^n\sum_{m=i}^n S_1(n,m)(-\rho)^{n-m}\binom{m}{i}\frac{(-z)^i}{[m-i+1]_q^k}\\
&=\sum_{i=0}^n\sum_{m=i}^{n+i}S_1(n,m)(-\rho)^{n-m}\binom{m}{i}\frac{(-z)^i}{[m-i+1]_q^k}\\
&=\sum_{i=0}^n\sum_{m=0}^{n}S_1(n,m+i)(-\rho)^{n-m-i}\binom{m+i}{i}\frac{(-z)^i}{[m+1]_q^k}\\
&=\sum_{m=0}^n\frac{(-\rho)^{n-m}}{[m+1]_q^k}\sum_{i=0}^m\binom{m+i}{i}\left(\frac{z}{\rho}\right)^i S_1(n,m+i)\\
&=\sum_{m=0}^n S_1\left(n,m,\frac{z}{\rho}\right)\frac{(-\rho)^{n-m}}{[m+1]_q^k}\,.
\end{align*} 
\qed\end{proof} 

\begin{Cor}  
For $q$-poly-Cauchy numbers of the first kind with a parameter $\rho$, we have 
\begin{equation} 
c_{n,\rho,q}^{(k)}=\sum_{m=0}^\infty S_1(n,m)\frac{(-\rho)^{n-m}}{[m+1]_q^k}\,. 
\label{exp:qpcn1}
\end{equation} 
\end{Cor}

Define the $q$-polyfactorial functions ${\rm Lif}_{k,q}(z)$ by 
$$
{\rm Lif}_{k,q}(z)=\sum_{n=0}^\infty\frac{z^n}{n![n+1]_q^k}\,. 
$$ 
Notice that 
$$
\lim_{q\to 1}{\rm Lif}_{k,q}(z)={\rm Lif}_k(z)\,,
$$
which is the ordinary polyfactorial function in (\ref{polyfac}).  
The generating function of $c_{n,\rho,q}^{(k)}$ is given by the following theorem. 

\begin{theorem}  
We have 
\begin{equation}  
\frac{1}{(1+\rho t)^{z/\rho}}{\rm Lif}_{k,q}\left(\frac{\ln(1+\rho t)}{\rho}\right)=\sum_{n=0}^\infty c_{n,\rho,q}^{(k)}(z)\frac{t^n}{n!}\,. 
\label{gen:qpc1}
\end{equation}  
\label{thc2} 
\end{theorem} 

\begin{proof}  
By the first identity of Theorem \ref{thc1}
\begin{align*} 
\sum_{n=0}^\infty c_{n,\rho,q}^{(k)}(z)\frac{t^n}{n!}&=\sum_{n=0}^\infty\sum_{m=0}^n S_1(n,m)(-\rho)^{n-m}\sum_{i=0}^m\binom{m}{i}\frac{(-z)^i}{[m-i+1]_q^k}\frac{t^n}{n!}\\
&=\sum_{m=0}^\infty(-\rho)^{-m}\sum_{n=m}^\infty S_1(n,m)\frac{(-\rho t)^n}{n!}\sum_{i=0}^m\binom{m}{i}\frac{(-z)^i}{[m-i+1]_q^k}\\
&=\sum_{m=0}^\infty\frac{1}{m!}\left(\frac{\ln(1+\rho t)}{\rho}\right)^m\sum_{i=0}^m\binom{m}{i}\frac{(-z)^i}{[m-i+1]_q^k}\\
&=\sum_{i=0}^\infty\frac{(-z)^i}{i!}\sum_{m=i}^\infty\frac{1}{(m-i)![m-i+1]_q^k}\left(\frac{\ln(1+\rho t)}{\rho}\right)^m\\
&=\sum_{n=0}^\infty\frac{\bigl(\ln(1+\rho t)\bigr)^n}{n! \rho^n}\frac{1}{[n+1]_q^k}\sum_{i=0}^\infty\frac{1}{i!}\left(-\frac{z\ln(1+\rho t)}{\rho}\right)^i\\
&=\frac{1}{(1+\rho t)^{z/\rho}}\sum_{n=0}^\infty\frac{\bigl(\ln(1+\rho t)\bigr)^n}{\rho^n}\frac{1}{n![n+1]_q^k}\\
&=\frac{1}{(1+\rho t)^{z/\rho}}{\rm Lif}_{k,q}\left(\frac{\ln(1+\rho t)}{\rho}\right)\,. 
\end{align*}  
\qed\end{proof}

\section{$q$-poly-Cauchy polynomials of the second kind with a parameter}

Let $n$ and $k$ be integers with $n\ge 0$ and $k\ge 1$, and let $\rho$ be a real number parameter with $\rho\ne 0$. Define $q$-poly-Cauchy polynomials of the second kind $\widehat c_{n,\rho,q}^{(k)}(z)$ with a parameter $\rho$ by 
\begin{align}  
\widehat c_{n,\rho,q}^{(k)}(z)&=\rho^n\underbrace{\int_0^1\cdots\int_0^1}_k\left(\frac{-x_1\cdots x_k+z}{\rho}\right)_n d_q x_1\cdots d_q x_k\notag\\
&=\rho^n n!\underbrace{\int_0^1\cdots\int_0^1}_k\binom{(-x_1\cdots x_k+z)/\rho}{n}d_q x_1\cdots d_q x_k\,. 
\label{def:qpc2}
\end{align}

Notice that 
$$
\lim_{q\to 1}\widehat c_{n,\rho,q}^{(k)}(z)=\widehat c_{n,\rho}^{(k)}(z)\,, 
$$ 
which is the poly-Cauchy polynomial of the second kind with a $\rho$ parameter (\cite{Ko2}). In addition, when $z=0$, $\widehat c_{n,\rho}^{(k)}(0)=\widehat c_{n,\rho}^{(k)}$ is the poly-Cauchy number of the second kind with a $\rho$ parameter (\cite{Ko2}). When $z=0$ and $\rho=1$, $\widehat c_{n,1}(0)=\widehat c_n^{(k)}$ is the poly-Cauchy number (\cite{Ko1}) given in (\ref{polycauchy2}). 

The $q$-poly-Cauchy polynomials of the first kind $\widehat c_{n,\rho,q}^{(k)}(z)$ with a parameter can be expressed in terms of the weighted Stirling numbers of the first kind $S_1(n,m,x)$.  In this expression, $k$ is not restricted to positive integers. 

\begin{theorem}  
For integers $n$ and $k$ with $n\ge 0$, we have 
\begin{align} 
\widehat c_{n,\rho,q}^{(k)}(z)&=(-1)^n\sum_{m=0}^n S_1(n,m)\rho^{n-m}\sum_{i=0}^m\binom{m}{i}\frac{(-z)^i}{[m-i+1]_q^k}\notag\\
&=(-1)^n\sum_{m=0}^n S_1\left(n,m,-\frac{z}{\rho}\right)\frac{\rho^{n-m}}{[m+1]_q^k}\,.  
\label{exp:qpc2}
\end{align}  
\label{thch1} 
\end{theorem} 
\begin{proof}
From (\ref{def:qpc1}) and (\ref{def:signedst1}), and using the relation \cite[Eq. (5.2)]{Ca2}, similarly to the proof of Theorem \ref{thc1}, we obtain the result. 
\qed\end{proof}    

\begin{Cor}  
For $q$-poly-Cauchy numbers of the second kind with a parameter $\rho$, we have 
$$
\widehat c_{n,\rho,q}^{(k)}=(-1)^n\sum_{m=0}^\infty S_1(n,m)\frac{\rho^{n-m}}{[m+1]_q^k}\,. 
$$
\end{Cor}
\begin{proof}
Putting $z=0$ in Theorem \ref{thch1}, we immediately get the result. 
\qed\end{proof}  

The generating function of $\widehat c_{n,\rho,q}^{(k)}$ is given by using the $q$-polyfactorial functions ${\rm Lif}_{k,q}(z)$. 

\begin{theorem}  
We have 
$$
(1+\rho t)^{z/\rho}{\rm Lif}_{k,q}\left(-\frac{\ln(1+\rho t)}{\rho}\right)=\sum_{n=0}^\infty\widehat c_{n,\rho,q}^{(k)}(z)\frac{t^n}{n!}\,. 
$$ 
\label{thch2} 
\end{theorem} 
\begin{proof}
Similarly to the proof of Theorem \ref{thc2}, by the first identity of Theorem \ref{thc1}, we obtain the result. 
\qed\end{proof}

\section{Several relations of $q$-poly-Bernoulli polynomials and $q$-poly-Cauchy polynomials}

There exist orthogonality and inverse relations for weighted Stirling numbers (\cite{Ca2}). Namely, from the orthogonal relations 
$$
\sum_{l=m}^n(-1)^{n-l}S_2(n,l,x)S_1(l,m,x)=\sum_{l=m}^n(-1)^{l-m}S_1(n,l,x)S_2(l,m,x)=\delta_{m,n}\,, 
$$ 
where $\delta_{m,n}=1$ if $m=n$; $\delta_{m,n}=0$ otherwise,  
we obtain the inverse relations 
\begin{equation} 
f_n=\sum_{m=0}^n(-1)^{n-m}S_1(n,m,x)g_m\quad\Longleftrightarrow\quad g_n=\sum_{m=0}^n S_2(n,m,x)f_m\,. 
\label{inverse} 
\end{equation}  

\begin{theorem} 
For $q$-poly-Bernoulli and $q$-poly-Cauchy polynomials with a parameter, we have
\begin{align} 
\sum_{m=0}^n S_1\left(n,m,\frac{z}{\rho}\right)\rho^{-m}B_{m,\rho,q}^{(k)}(z)&=\frac{n!}{\rho^n[n+1]_q^k}\,,
\label{201}\\
\sum_{m=0}^n S_2\left(n,m,\frac{z}{\rho}\right)\rho^{-m}c_{m,\rho,q}^{(k)}(z)&=\frac{1}{\rho^n[n+1]_q^k}\,,
\label{202}\\
\sum_{m=0}^n S_2\left(n,m,-\frac{z}{\rho}\right)\rho^{-m}\widehat c_{m,\rho,q}^{(k)}(z)&=\frac{(-1)^n}{\rho^n[n+1]_q^k}\,.
\label{203}
\end{align}
\label{th20}
\end{theorem}  

\noindent 
{\it Remark.}  
If $q\to 1$, then Theorem \ref{th20} is reduced to Theorem 3.2 in \cite{KC}.  

\begin{proof}  
By Theorem \ref{thb1}, applying (\ref{inverse}) with 
$$
f_m=\frac{m!}{(-\rho)^m[m+1]_q^k},\quad g_n=\frac{B_{n,\rho,q}^{(k)}(z)}{(-\rho)^n}
$$ 
and $x$ is replaced by $z/\rho$, 
we get the identity (\ref{201}).
Similarly, by Theorem \ref{thc1} and Theorem \ref{thch1} we have the identities (\ref{202}) and (\ref{203}), respectively. 
\qed\end{proof} 

There are relations between two kinds of $q$-poly-Cauchy polynomials with a parameter.  

\begin{theorem}  
For $n\ge 1$ we have 
\begin{align}  
(-1)^n\frac{c_{n,\rho,q}^{(k)}(z)}{n!}&=\sum_{m=1}^n\binom{n-1}{m-1}\frac{\widehat c_{m,\rho,q}^{(k)}(z)}{m!}\,,
\label{301}\\
(-1)^n\frac{\widehat c_{n,\rho,q}^{(k)}(z)}{n!}&=\sum_{m=1}^n\binom{n-1}{m-1}\frac{c_{m,\rho,q}^{(k)}(z)}{m!}\,. 
\label{302} 
\end{align}  
\label{th30}
\end{theorem}  

\noindent 
{\it Remark.} 
Since $c_{n,1,q}^{(k)}(z)\to c_n^{(k)}(z)$ and $\widehat c_{n,1,q}^{(k)}(z)\to\widehat c_n^{(k)}(z)$ as $q\to 1$, Theorem \ref{th30} is reduced to Theorem 4.2 in \cite{KK}.  

\begin{proof}  
We shall prove identity (\ref{302}). The identity (\ref{301}) can be proven similarly. By the definition of $\widehat c_{n,q}^{(k)}(z)$ 
\begin{align*}  
(-1)^n\frac{\widehat c_{n,\rho,q}^{(k)}(z)}{n!}&=(-\rho)^n\underbrace{\int_0^1\cdots\int_0^1}_k\binom{(-x_1\cdots x_k+z)/\rho}{n}d_q x_1\dots d_q x_k\\
&=\rho^n\underbrace{\int_0^1\cdots\int_0^1}_k\binom{(x_1\cdots x_k-z)/\rho+(n-1)}{n}d_q x_1\dots d_q x_k\,.
\end{align*} 
We use the well-known identity, see \cite[p. 8]{Rio},   
\begin{equation}  
\binom{x+y}{n}=\sum_{l=0}^n\binom{x}{l}\binom{y}{n-l}
\label{eq:62}  
\end{equation} 
as $x=(x_1\cdots x_k-z)/\rho$ and $y=n-1$.  
Then 
\begin{align*} 
(-1)^n\frac{\widehat c_{n,\rho,q}^{(k)}(z)}{n!}&=\rho^n\underbrace{\int_0^1\cdots\int_0^1}_k\sum_{m=0}^l\binom{(x_1\cdots x_k-z)/\rho}{m}\binom{n-1}{n-m}d_q x_1\dots d_q x_k\\
&=\sum_{m=0}^n\binom{n-1}{m-1}\rho^n\underbrace{\int_0^1\cdots\int_0^1}_k\binom{(x_1\cdots x_k-z)/\rho}{m}d_q x_1\dots d_q x_k\\
&=\sum_{m=1}^n\binom{n-1}{m-1}\frac{c_{m,\rho,q}^{(k)}(z)}{m!}\,. 
\end{align*} 
Note that $\binom{n-1}{-1}=0$. 
\qed\end{proof} 

\begin{theorem}
For any $x$ and $y$ we have
\begin{align*}
B_{n,\rho,q}^{(k)}(x)
&=\sum_{l=0}^{n}\sum_{m=0}^{n}(-1)^{n-m}m!\rho^{n-l}S_{2}\left(n,m,\frac{x}{\rho}\right)S_{2}\left(m,l,\frac{y}{\rho}\right)c_{l,\rho,q}^{(k)}(y)\,,\\
B_{n,\rho,q}^{(k)}(x)
&=\sum_{l=0}^{n}\sum_{m=0}^{n}(-1)^{n}m!\rho^{n-l}S_{2}\left(n,m,\frac{x}{\rho}\right)S_{2}\left(m,l,-\frac{y}{\rho}\right)\widehat c_{l,\rho,q}^{(k)}(y)\,,\\
c_{n,\rho,q}^{(k)}(x)  
&=\sum_{l=0}^{n}\sum_{m=0}^{n}
\frac{(-1)^{n-m}}{m!}\rho^{n-l}S_{1}\left(n,m,\frac{x}{\rho}\right)S_{1}\left(m,l,\frac{y}{\rho}\right)B_{l,\rho,q}^{(k)}(y)\,,\\
\widehat{c}_{n,\rho,q}^{(k)}(x)
&=\sum_{l=0}^{n}\sum_{m=0}^{n}\frac{(-1)^{n}}{m!}
\rho^{n-l}S_{1}\left(n,m,-\frac{x}{\rho}\right)S_{1}\left(m,l,\frac{y}{\rho}\right)B_{l,\rho,q}^{(k)}(y)\,.
\end{align*}
\label{th50} 
\end{theorem}  

\noindent 
{\it Remark.}  
If $\rho=1$ and $q\to 1$, then Theorem \ref{th50} is reduced to Theorem 3.3 in \cite{KC}. 
If $\rho=1$, $x=y$ and $q\to 1$, then Theorem \ref{th50} is reduced to Theorem 4.1 in \cite{KL}. A different generalization without Jackson's integrals is discussed in \cite{KLS}.  

\begin{proof}  
We shall prove the first and the fourth identities. Others can be proven similarly. By (\ref{202}) in Theorem \ref{th20}, and using (\ref{exp:qpb}), we have 
\begin{align*}  
&\sum_{l=0}^{n}\sum_{m=0}^{n}(-1)^{n-m}m!\rho^{n-l}S_{2}\left(n,m,\frac{x}{\rho}\right)S_{2}\left(m,l,\frac{y}{\rho}\right)c_{l,\rho,q}^{(k)}(y)\\
&=\sum_{m=0}^n(-1)^{n-m}m!\rho^n S_2\left(n,m,\frac{x}{\rho}\right)\sum_{l=0}^m S_2\left(m,l,\frac{y}{\rho}\right)\rho^{-l} c_{l,\rho,q}^{(k)}(y)\\
&=\sum_{m=0}^n S_2\left(n,m,\frac{x}{\rho}\right)\frac{(-\rho)^{n-m}m!}{[m+1]_q^k}\\
&=B_{n,\rho,q}^{(k)}(x)\,. 
\end{align*} 
By (\ref{201}) in Theorem \ref{th20}, and using (\ref{exp:qpc2}), we have 
\begin{align*}  
&\sum_{l=0}^{n}\sum_{m=0}^{n}\frac{(-1)^{n}}{m!}
\rho^{n-l}S_{1}\left(n,m,-\frac{x}{\rho}\right)S_{1}\left(m,l,\frac{y}{\rho}\right)B_{l,\rho,q}^{(k)}(y)\\
&=\sum_{m=0}^n\frac{(-1)^n}{m!}\rho^n S_1\left(n,m,-\frac{x}{\rho}\right)\sum_{l=0}^m S_1\left(m,l,\frac{y}{\rho}\right)\rho^{-l}B_{l,\rho,q}^{(k)}(y)\\
&=(-1)^n\sum_{m=0}^n S_1\left(n,m,-\frac{x}{\rho}\right)\frac{\rho^{n-m}}{[m+1]_q^k}\\
&=\widehat c_{n,\rho,q}^{(k)}(x)\,. 
\end{align*} 
\qed\end{proof}

\section{Acknowledgement}  

This work was supported in part by the grant of Wuhan University and by the grant of Hubei Provincial Experts Program. 

\end{document}